\documentclass[11pt,a4paper]{article}

\usepackage{indentfirst,csquotes}
\usepackage{authblk}

\usepackage[bookmarks=false]{hyperref}
   \hypersetup{colorlinks=true,
 linkcolor=blue!80,
 citecolor=orange!70!black,
 urlcolor=blue!80}

\usepackage{a4wide}

\usepackage{amsthm,mathtools}
\usepackage{amsmath,amssymb}
\usepackage{amsfonts}
\usepackage[thmtools-compat]{keytheorems}
\usepackage{zref-clever}
\zcsetup{
       cap = true,
    abbrev = false,
    nameinlink,
}

\usepackage{xcolor,paralist,titlesec,fancyhdr,etoolbox}
\usepackage[figuresright]{rotating}

\newtheorem{theorem}{Theorem}[section]
\newtheorem{lemma}[theorem]{Lemma}

\newtheorem{corollary}[theorem]{Corollary}

\newtheorem{proposition}[theorem]{Proposition}
\newtheorem{remark}[theorem]{Remark}
\usepackage{tikz}

\begin{document}
\title{On $\{2\}$-Roman graph recognition of Partner Limited graphs}

\author[1,2]{Lara Fernández\thanks{lara@fceia.unr.edu.ar}$^,$}
\author[1,2]{Valeria Leoni}
\affil[1]{FCEIA, Universidad Nacional de Rosario, Argentina}
\affil[2]{CONICET, Argentina}

\maketitle

\begin{abstract}
Given a graph $G$ with vertex set $V$, $f : V \rightarrow \{0, 1, 2\}$ is a \emph{Roman $\{2\}$-dominating function} (or \emph{italian dominating function}) of $G$ if for every vertex $v\in V$ with $f(v) =0$, either there exists a vertex $u$ adjacent to $v$ with $f(u) = 2$, or two distinct vertices $x,\; y$ both adjacent to $v$ with $f(x)=f(y)=1$ (Chellali et al. 2016). Every graph $G$ satisfies $\gamma_{\{R2\}}(G) \leq 2\gamma(G)$, where $\gamma_{\{R2\}}(G)$ denotes the minimum weight of a $\{2\}$-Roman dominating function of $G$ and $\gamma(G)$ is the domination number of $G$.   \emph{$\{2\}$-Roman graphs}  are those for which the equality is reached (Klostermeyer et al. 2019). A  characterization  ---in terms of four graph operations--- of $\{2\}$-Roman trees was  given by  Henning et al. in 2017. In 2025, Ferrari et al. characterized the $\{2\}$-Roman property  by the existence of a minimum $\{2\}$-Roman dominating function of $G$ that assumes only $0, 2$-values. Afterwards in 2025,  Bešter Štorgel et al. introduced  the problem of recognizing $\{2\}$-Roman graphs,  proved polinomiality  for middle graphs, and characterized \hbox{$\{2\}$-Roman} split graphs that can be decomposed with respect to the split join operation into two smaller split graphs.  Recognition complexity is still open for general graphs. 
     
     In this paper we contribute to this line of research. We study the \hbox{$\{2\}$-Roman}  property on graphs that can be decomposed into two smaller graphs with respect to the join and union operations, allowing to completely characterize the $\{2\}$-Roman property.  The 4-path is the  non trivial connected non $\{2\}$-Roman graph with the fewest number of vertices and edges.  We  classify the $\{2\}$-Roman property within specific families of non decomposable graphs  with a limited number of 4-paths which are present in the decomposition of partner limited graphs; these are well-labelled spiders, the graphs in ZOO and some special split graphs. 
\end{abstract}

\section{Introduction, preliminaries and definitions}

All graphs in this paper are undirected and simple. Let $G$ be a graph, and let $V(G)$ and $E(G)$ denote its vertex and edge sets, respectively.
Whenever it is clear from the context, we simply write $V$ and $E$. 

For a graph $G=(V,E)$, the \emph{complementary} graph of $G$, or complement graph, is the graph
$\overline{G}=(V, \overline{E})$, where $\overline{E}=\{uv: u, v \in V, \; uv\notin E\}$.
Two vertices of $V$ are \emph{adjacent} in $G$ if there is an edge of $E$ that includes both. The \emph{open neighborhood} of $v\in V$, $N_G(v)$, is the set of all vertices adjacent to $v$ in $G$, that is, $N_G(v)=\{w: vw\in E\}$. The \emph{closed neighborhood} of $v\in V$, $N_G[v]$ is the set $N_G(v)\cup \{v\}$. For a vertex $v\in V$, the number of vertices to which $v$ is adjacent is the \emph{degree} of $v$ in $G$, denoted by $deg_G(v)$. The minimum degree over the vertices of $G$ is denoted by $\delta(G)$.

A \emph{pendant} vertex is a vertex of degree one. A \emph{universal} vertex is a vertex of degree $|V|-1$.

An \emph{n-path} (\emph{n-cycle}) is a path (cycle) on $n$ vertices and is denoted by $P_n$ ($C_n$). 

A graph is \emph{split} if its vertex set is partitioned into $S$ and $K$, where $S$ is a stable set and $K$ is a clique. For a split graph we use the notation $G=(S,K)$. Split graphs can be recognized in linear time from \cite{Hammer1981}.

Two graphs $G$ and $H$ are \emph{isomorphic} if there is a bijection $f$ between $V(G)$ and $V(H)$
 such that $uv \in E(G)$ if and only if $f(u)f(v) \in E(H)$.
 
Given a graph $G$ and $S\subseteq V$, the \emph{subgraph induced by $S$} is the graph with vertex set $S$ and such that two vertices of $S$ are adjacent if and only if they are adjacent in $G$. Besides, $G \setminus S$ denotes the subgraph of $G$ induced by $V\setminus S$. Given a set of graphs $\mathcal{H}$, we say that $G$ is $\mathcal{H}$-free if $G$ has no graph in $\mathcal{H}$ as induced subgraph. 

Given two graphs $G_1$ and $G_2$, the \emph{union} of $G_1$ and $G_2$, denoted by $G_1 \cup G_2$, is the binary operation that returns the graph with vertex set $V(G_1) \cup V(G_2)$ and edge set $E(G_1) \cup E(G_2)$, and the \emph{join} of $G_1$ and $G_2$, denoted by $G_1 \wedge G_2$, is the binary operation that returns the graph with vertex set $V(G_1) \cup V(G_2)$ and edge set $E(G_1) \cup E(G_2)$ together with all edges with one endpoint in $G_1$ and the other in $G_2$.

It is known that, for a graph $G$, if $G$ or the complementary
graph of $G$ is not connected, $G$ is the union or the join of
two smaller graphs, respectively. This linear decomposition procedure can
be performed iteratively on every connected component of $G$ or its complementary graph, until all subgraphs and their complementary graphs are connected. For details, see for instance \cite{Habib1994}. If both $G$ and $\overline{G}$ are connected, we will say that $G$ is \emph{non decomposable}.

 Given a graph $G=(V,E)$, $f : V \rightarrow \{0, 1, 2\}$ is a \emph{Roman $\{2\}$-dominating function} (or \emph{italian dominating function}) of $G$ if for every vertex $v\in V$ with $f(v) =0$, either there exists a vertex $u \in V$ adjacent to $v$ with $f(u) = 2$, or two distinct vertices $x, y \in V$ both adjacent to $v$ with $f(x)=f(y)=1$ \cite{Chellali2016}. The minimum weight of a Roman $\{2\}$-dominating function of $G$ is called the \emph{Roman $\{2\}$-domination number} of $G$ and is denoted by 
$\gamma_{\{R2\}}(G)$, where the \emph{weight} of a function $f$ defined on $V$ is the value $f(V) = \sum_{v\in V} f(v)$.

Every graph $G$ satisfies $\gamma_{\{R2\}}(G) \leq 2\gamma(G)$, where $\gamma_{\{R2\}}(G)$ denotes the minimum weight of a $\{2\}$-Roman dominating function of $G$ and $\gamma(G)$ is the domination number of $G$. Graphs for which equality is attained are called \emph{$\{2\}$-Roman graphs} \cite{Klostermeyer}. It is not hard to notice that $G$ is a $\{2\}$-Roman graph if and only if every connected component is a $\{2\}$-Roman graph. For a graph $G$, $\gamma_{\{R2\}}(G)=1$ if and only if $|V(G)|=1$ and thus, the graph with only one vertex is not $\{2\}$-Roman. Also, $\gamma_{\{R2\}}(G)=2$ if and only if $G$ has a universal vertex or if $G$ has two vertices that are adjacent to every other vertex, but each other. 
The $\{2\}$-Roman property of graphs is characterized in \cite{Ferrari2025} by the existence of a minimum $\{2\}$-Roman dominating function of $G$ that assumes only $0, 2$-values. In \cite{vStorgel2025}, Bešter Štorgel et al. introduce for the first time the problem of recognizing $\{2\}$-Roman graphs and prove that it is polynomial for  middle graphs, but recognition remains open for general  graphs. They also characterize \hbox{$\{2\}$-Roman} split graphs that can be decomposed with respect to the split join operation into two smaller split graphs. 

It is not hard to check that the smallest connected non trivial non $\{2\}$-Roman graph is $P_4$. Although the $\{2\}$-Roman property  is not inherited by induced subgraphs, in this work with the purpose of delve deeper into this open study, we explore graph classes that restrict the number of induced 4-paths.  
In Section \ref{sec:operations}, we analyze  the behavior of Roman $\{2\}$-domination on the join and union of graphs and then  completely characterize the $\{2\}$-Roman property  with respect to the two smaller graphs involved in these operations. In Section~\ref{sec:special}, we  obtain closed formulas for the Roman $\{2\}$-domination number of  the non decomposable  graphs  ---with respect to the union and join operation--- that are present in the decomposition  of partner-limited graphs.
In Section \ref{sec:sec5}, we characterize  the $\{2\}$-Roman property among the non decomposable graphs addressed in Section~\ref{sec:special}. 

\section{$\{2\}$-Roman union and join graphs}\label{sec:operations}

Firstly, from the definition of the union of two graphs, it holds $\gamma_{\{R2\}}(G_1\cup G_2)=\gamma_{\{R2\}}(G_1)+\gamma_{\{R2\}}(G_2)$ for any graphs $G_1$ and $G_2$. Secondly, for the join of two graphs, it is clear that $\gamma_{\{R2\}}(G_1\wedge G_2)\leq \min\{\gamma_{\{R2\}}(G_1), \gamma_{\{R2\}}(G_2)\}$ for any graphs $G_1$ and $G_2$. We next show how to obtain the exact value of $\gamma_{\{R2\}}(G_1\wedge G_2)$.

We use the notation $f=(V_0, V_1, V_2)$ to represent a Roman $\{2\}$-dominating function $f$ of $G$, where $V_i$ is the subset of vertices $v\in V$ such that $f(v)=i$.
Besides, for simplicity, we say that $f$ is a 
$\gamma_{\{R2\}}(G)$-function when $f$ is a Roman $\{2\}$-dominating function of $G$ with ---minimum--- weight $\gamma_{\{R2\}}(G)$.  

For a positive integer $r$, we  will use the notation $[1,r]$ to indicate the finite set of integer numbers between 1 and $r$, i.e the set $\{1, \ldots, r\}$.

\begin{theorem}\label{join}
Let $G_1$ and $G_2$ be two graphs with $|V(G_i)|\geq 2$ for $i=1,2$ and \[m=\min\{\gamma_{\{R2\}}(G_1), \gamma_{\{R2\}}(G_2)\}.\] Then 
\begin{enumerate}
  \item $\gamma_{\{R2\}}(G_1\wedge G_2)=2$ if and only if $m= 2$;
  
  \item $\gamma_{\{R2\}}(G_1\wedge G_2)=3$ if and only if $m=3$ or, $m=4$ and $\gamma(G_i)=2$ for some $i=1,2$;
  \item $\gamma_{\{R2\}}(G_1\wedge G_2)=4$ if and only if $m\geq 4$ and $\gamma(G_i)>2$ for $i=1,2$.
\end{enumerate}
\end{theorem}

\begin{proof}
\begin{enumerate}
  \item Let $\gamma_{\{R2\}}(G_1\wedge G_2)=2$ and $g=(V_0,V_1,V_2)$ be a $\gamma_{\{R2\}}(G_1\wedge G_2)$-function. Since $|V(G_i)|\geq 2$ for $i=1,2$, then $m\geq 2$. 
  If $V_2=\{v\}$ for some $v\in V(G_1) \cup V(G_2)$, then every other vertex in $V(G_1)\cup V(G_2)$ is adjacent to $v$. Then $v$ is a universal vertex in $G_1\wedge G_2$. Moreover, $v$ is a universal vertex in $G_i$ for $i=1$ or $i=2$, and thus $\gamma_{\{R2\}}(G_i)= 2$ for that $i$, implying $m\leq2$. 
If $V_1=\{u,v\}$ and $u,v\in V(G_i)$ for some $i$, then every vertex in $V(G_i)\setminus \{u,v\}$ is adjacent to both $u$ and $v$. Hence $\gamma_{\{R2\}}(G_i)=2$, implying $m\leq2$.  Otherwise ---$u\in V(G_1)$ and $v\in V(G_2)$--- every vertex in $G_1$ is adjacent to $u$ and every vertex in $G_2$ is adjacent to $v$. Thus both $u$ and $v$ are universal vertices in $G_1$ and $G_2$ respectively, concluding that $\gamma_{\{R2\}}(G_1)= \gamma_{\{R2\}}(G_2)= 2$.
The converse is easy, since $1< \gamma_{\{R2\}}(G_1\wedge G_2)\leq \gamma_{\{R2\}}(G_i)$ for each $i$, $\gamma_{\{R2\}}(G_1)= 2$ (or $\gamma_{\{R2\}}(G_2)= 2$) implies $\gamma_{\{R2\}}(G_1\wedge G_2)=2$.
 \item Let $\gamma_{\{R2\}}(G_1\wedge G_2)=3$ and $g=(V_0,V_1,V_2)$ be a $\gamma_{\{R2\}}(G_1\wedge G_2)$-function. 
 From the previous item, $\gamma_{\{R2\}}(G_1)> 2$ and $\gamma_{\{R2\}}(G_2)> 2$. 
\begin{itemize}
  
\item When $V_2\neq \emptyset$ ($V_2=\{v\}$ for some $v$) and then $V_1=\{u\}$ for some $u$, then w.l.o.g. $\{u,v\} \subseteq V(G_1)$ (or $\{u,v\} \subseteq V(G_2)$), otherwise ($v\in V(G_1)$ and $u\in V(G_2)$) every vertex of $G_1$ distinct from $v$ is adjacent to $v$ implying that $v$ is universal in $G_1$, thus $\gamma_{\{R2\}}(G_1)=2$ which is a contradiction. Thus, the restriction of $g$ to $V(G_1)$ is a Roman $\{2\}$-dominating function of $G_1$, implying $\gamma_{\{R2\}}(G_1)=3$, and thus $m=3$.
 
 \item When $V_1=\{u,v,w\}$ and $u,v,w\in V(G_i)$ for $i=1$ or $i=2$, then again the restriction of $g$ to $V(G_i)$ is a Roman $\{2\}$-dominating function of $G_i$, implying $\gamma_{\{R2\}}(G_i)=3$ and thus $m=3$.
 
 \item When $V_1=\{u,v,w\}$ (w.l.o.g. $u,v \in V(G_1)$ and $w\in V(G_2)$), since every vertex in $G_1$ is adjacent to $w$, then every vertex in $G_1$ must be adjacent also to either $u$ or $v$. Then $\{u,v\}$ is a dominating set in $G_1$, concluding that $\gamma(G_1)= 2$ ($\gamma(G_1)$ cannot be equal to 1 since in this case $G_1$ would have a universal vertex implying $\gamma_{\{R2\}}(G_1)=2$, a contradiction).
 \end{itemize}
 
For the converse, assume first that $m=3$ (w.l.o.g. $\gamma_{\{R2\}}(G_1)=3$). Then a minimum $\gamma_{\{R2\}}(G_1)$-function is a Roman $\{2\}$-dominating function of $G_1\wedge G_2$ implying $\gamma_{\{R2\}}(G_1\wedge G_2)\leq 3$.

Now assume $m>3$. If $\gamma(G_1)=2$ ($\gamma(G_2)=2$), then $\{u,v\}$ is a dominating set for $G_1$ ($G_2$). Take any vertex $w\in V(G_2)$ ($V(G_1)$). Note that every vertex in $V(G_1)$ is adjacent in $G_1\wedge G_2$ to either $u$ or $v$, and also adjacent to $w$. In addition, every vertex in $V(G_2)$ is adjacent to both, $u$ and $v$. Then the function that assigns 1 to $u$, $v$ and $w$, and assigns 0 to every other vertex in $V(G_1\wedge G_2)$ is a Roman $\{2\}$-dominating function of $G_1\wedge G_2$ with weight 3, implying $\gamma_{\{R2\}}(G_1\wedge G_2)\leq 3$.
  
But $\gamma_{\{R2\}}(G_1\wedge G_2) >2 $ otherwise $G_1\wedge G_2$, and also $G_i$ for $i=1,2$, would have a universal vertex implying $\gamma_{\{R2\}}(G_1)=\gamma_{\{R2\}}(G_2)=2 <m$, a contradiction.  
  \item $\gamma_{\{R2\}}(G_1\wedge G_2)\geq 4$ follows from the previous items. Take $u \in V(G_1)$ and $v \in V(G_2)$ and define the function $f=(V_0=\emptyset,V_1=\emptyset,V_2=\{u,v\})$. It turns out that $f$ is a Roman $\{2\}$-dominating function of $G_1\wedge G_2$ since every vertex in $V(G_1)$ is adjacent to $v$ and every vertex in $V(G_2)$ is adjacent to $u$. Since $f$ has weight 4, we arrive at $\gamma_{\{R2\}}(G_1\wedge G_2)=4$. 
\end{enumerate}
\end{proof}

By recalling from the introduction that a non connected graph is a $\{2\}$-Roman graph if and only if every connected component of it is a $\{2\}$-Roman graph, for a graph that is the union of two other graphs we  have:

\begin{theorem}
    Let $G_1$ and $G_2$ be two graphs. Then $G_1 \cup G_2$ is a $\{2\}$-Roman graph if and only if $G_1$ and $G_2$ are both $\{2\}$-Roman graphs. 
    \end{theorem}

For the join of two other graphs we can state:

\begin{theorem}\label{th:2rjoin}
    Let $G_1$ and $G_2$ be two  graphs. Then $G_1\wedge G_2$ is a $\{2\}$-Roman graph if and only if one of the following conditions holds:
    \begin{enumerate}
        \item $G_1$ or $G_2$ has a universal vertex;
        \item $G_1$ and $G_2$ are non connected with $|V(G_i)|\geq 3$ and $\gamma(G_i)\geq 3$ for $i=1,2$;
         \item $G_1$ ($G_2$) is non decomposable and  $G_2$ ($G_1$) is  non connected with
        $\gamma(G_i)\geq 3$ and $\gamma_{\{R2\}}(G_i)\geq 4$ for $i=1,2$;
        \item $G_1$ and $G_2$ are both non decomposable with
        $\gamma(G_i)\geq 3$ and $\gamma_{\{R2\}}(G_i)\geq 4$ for $i=1,2$.

    \end{enumerate} 
\end{theorem}

\begin{proof}

By taking into account Theorem \ref{join}, the join $G_1\wedge G_2$ is a $\{2\}$-Roman graph if and only if:
\begin{itemize}
\item[Case 1:] $\gamma(G_1\wedge G_2)=1$ and $\gamma_{\{R2\}}(G_1\wedge G_2)=2$, or
\item [Case 2:] $\gamma(G_1\wedge G_2)=2$ and $\gamma_{\{R2\}}(G_1\wedge G_2)=4$.
\end{itemize}

    Case 1 occurs if and only if any  $G_1$ or $G_2$ has a universal vertex. 
    
    Then suppose that $G_1$ nor $G_2$ has a universal vertex. Then clearly $\gamma(G_1\wedge G_2)=2$ and, for  
    Case 2 to happen,   we. must have $\gamma_{\{R2\}}(G_i)\geq 4$ and $\gamma(G_i)>2$ for $i=1,2$.

    Let us assume --- w.l.o.g.--- that $G_1$ is connected and $\overline{G_1}$ is non connected, then $G_1$ is the join of two other graphs with no universal vertex.   
    It follows that $\gamma(G_1)=2$, thus we are not in Case 2.

    Now, let us consider that $G_1$ and $ G_2$ are non connected graphs, then $\gamma(G_1)\geq 2$ and $\gamma(G_2)\geq 2$. If $\gamma(G_i)=2$ for some $i=1,2$, $\gamma(G_1\wedge G_2)=2$ while $\gamma_{\{R2\}}(G_1\wedge G_2)=3$. Then $G_1\wedge G_2$ is not a $\{2\}$-Roman graph. Now let $\gamma(G_1)\geq 3$ and $\gamma(G_2)\geq 3$. Then $\gamma(G_1\wedge G_2)=2$ and $\gamma_{\{R2\}}(G_1\wedge G_2)=4$. Therefore $G_1\wedge G_2$ is a $\{2\}$-Roman graph.

    The remaining cases to be analyzed for which we are also in Case 2 are:  $G_1$ ($G_2$) is non decomposable, and $G_2$ ($G_1$) is non decomposable or non connected.
    \end{proof}

\begin{corollary}
    Let $G_1$ and $G_2$ be two graphs, then it is  polynomial time recognizable if $G_1\wedge G_2$ is $\{2\}$-Roman. 
\end{corollary}

\section{Roman $\{2\}$-domination on non decomposable graphs with ``few'' 4-paths}\label{sec:special}

In this section we first obtain closed formulas for the Roman $\{2\}$-domination number of the non decomposable  graphs among partner-limited graphs.  

Let $G$ be a graph, and let $A$ be an induced $P_4$ in $G$. 
A \emph{partner} of $A$ is a vertex $v$ in $G\setminus A$ such that $A\cup \{u\}$ induces a graph with at least two induced 4-paths. 
$G$ is a \emph{partner-limited} graph (or $PL$-graph) if any induced $P_4$ in $G$ has at most two partners. From \cite{roussel1999}, a graph $G$ is partner-limited if and only if exactly one of the following conditions
holds: $G$ or $\overline{G}$ is not connected and each one of its connected components is a partner-limited graph, $G$ is a well labelled spider, $G$ is isomorphic to a graph in $ZOO$, $G$ is a clique or a stable graph or $G$ is a split $\{H_1,H_2,\overline{H_1},\overline{H_2}\}$-free graph, where $H_1$ and $H_2$ are shown in Figure \ref{Hs}.

\subsection{Spiders}

A graph $G$ is a \emph{spider} if its vertex set can be partitioned into $(S,C,H)$ such that every vertex in $H$ is adjacent to every vertex in $C$ and to no vertex in $S$, $C$ is a clique and $S$ is a stable set with $S=\{s_1,\ldots,s_r\}$, $C=\{c_1,\ldots,c_r\}$ and $r\geq 2$, and such that only one of the following conditions hold:  $s_i$ is adjacent to $c_j$ if and only if $i=j$ (in which case $G$ is called a \emph{thin spider} and an edge $s_ic_i$ a \emph{leg}), or $s_i$ is adjacent to $c_j$ if and only if $i\neq j$ (in which case $G$ is called a \emph{thick spider}). It is straightforward that the complement graph of a thin spider is a thick spider and vice-versa. $(S,C,H)$ is called the \emph{(spider) partition} of the spider and can be found in linear time \cite{Jamison1992}.

\begin{proposition}\label{spider}
Let $G$ be a spider, then
$\gamma_{\{R2\}}(G)= 3$ if $G$ is thick, and $\gamma_{\{R2\}}(G)=|S|+1$ if $G$ is thin.
\end{proposition}

\begin{proof} 
Let $G$ be a spider with spider partition $(S,C,H)$ and $S=\{s_1,\ldots,s_r\}$, $C=\{c_1,\ldots,c_r\}$ and $r\geq 2$.

If $G$ is thick, we define on $V$ the function $f=(V\setminus\{c_j,s_j,\}, \{s_j\}, \{c_j\})$, where $j$ is any index in $[1,r]$. Since $c_j$ is adjacent to every vertex in $H\cup C$ and also to every vertex in $S$ but $s_j$, then $f$ is a Roman \{2\}-dominating function of $G$ with weight 3,  thus $\gamma_{\{R2\}}(G)\leq 3$. To prove the other inequality, let us assume that there exists a Roman \{2\}-dominating function $g$ of $G$ with weight 2. Thus $g=(V\setminus \{v\},\emptyset,{v})$ or $g=(V\setminus \{u,v\},\{u,v\},\emptyset )$ for some distinct vertices $u$ and $v$. In the first case, every vertex distinct from $u$ and $v$ in $G$ must be adjacent to $v$, which is not possible since $G$ has no universal vertex. In the second case, every other vertex must be adjacent to both $u$ and $v$. The only possibility for $u$ and $v$ is that they both belong to $C$, therefore $u=c_i$ and $v=c_j$ for some $i,j\in [1, r]$. But $s_i$ is adjacent to only one vertex in $V_1$, this is $c_j$ (the same happens to $s_j$), contradicting the fact that $g$ is a Roman \{2\}-dominating function of $G$. Thus, $\gamma_{\{R2\}}(G) \geq 3$.

If $G$ is thin, note that $G$ can be obtained by adding a pendant vertex to every vertex in $S$. Since every vertex in $H$ is adjacent to every vertex in $S$, we are under the hypothesis of Proposition \ref{stem} and $\gamma_{\{R2\}}(G)\geq |S|+1$. Then we define on $G$ the function $f=( H\cup (C\setminus \{c_j\})\cup \{s_j\},S\setminus \{s_j\},\{c_j\} )$ which is clearly a Roman \{2\}-dominating function of $G$, where again $j$ is any index in $[1,r]$. Since $f(G)=|S|+1$, we have $\gamma_{\{R2\}}(G)= |S|+1$. 
\end{proof}

\subsection{Well labelled spiders}

Although it is not the original definition given in \cite{roussel1999}, we consider the one given in \cite{argiroffo2014}.
  A graph $W$ is a \emph{well labelled spider} if and only if $W$ is obtained from a thin spider $(S,C,H)$ by performing one of the operations R1, R2, R3 once, where:
  
\begin{itemize}
\item R1: replace one vertex $v\in S$ by a graph with three vertices; 
\item R2: replace both endpoints of one leg of $(S,C,H)$ by graphs with two vertices;
\item R3: replace at most one endpoint of every leg of $(S,C,H)$ by a graph with two vertices. 
\end{itemize}

Next, we consider the addition of pendant vertices to a given graph. When we add a pendant (new) vertex $w$ to a vertex $v$ of a graph $G$,  we mean that we build a graph $H^v$ with vertex set $V(G)\cup \{w\}$ and edge set $E(H)=E(G) \cup \{vw\}$. We can prove:

\begin{proposition}\label{stem}
   Given a graph $G=(V,E)$, let $S$ be a subset of $V$ and consider the graph $H^S$ obtained from $G$ by adding a pendant vertex to each vertex of $S$. Then $\gamma_{\{R2\}}(G)\geq |S|+1$.
\end{proposition}
\begin{proof}
  Let us call, respectively, $v_i$ and $u_i$, the pendant vertex in $H^S$ and its neighbor for $i\in [1,|S|]$, and let $f$ be a $\gamma_{\{R2\}}(H^S)$-function. From the definition, $f(N[w])\geq 1$ for every $w\in V(H^S)$, in particular $f(N[v_i])\geq 1$ for every $i\in [1,|S|]$. Since $v_i$'s closed neighbourhood are pairwise disjoint, $f(H^S)\geq \sum\limits_{i=1}^{|S|} f(N[v_i]) \geq |S|$.
  Now assume $\gamma_{\{R2\}}(H^S)=|S|$. Then $f(v_i)=1$ and $f(u_i)=0$ and thus $f(N(u_i))=1$ for every $i \in [ 1, |S|]$, which contradicts that $f$ is a Roman $\{2\}$-dominating  function of $H^S$. Therefore, $\gamma_{\{R2\}}(G)\geq |S|+1$.
\end{proof}

\medskip

Finally, we consider the graph operation that, for a given vertex $v\in V$, adds to $G$ a new vertex $\tilde{v}$ in such a way that $v$ and $\tilde{v}$ are twins in the resulting graph $G\sim v$.
We say that $v$ and $\tilde{v}$ are \emph{twins} in $G$ if $N_G(v) \setminus \{\tilde{v}\}= N_G(\tilde{v})\setminus \{{v}\}$. Besides, when twins $v$ and $\tilde{v}$ are adjacent in $G$, we say that they are \emph{true twins} and when they are not, we say that they are \emph{false twins}. We can prove:
 
 \begin{proposition}\label{twins}
  Let $G$ be a graph, and $v \in V(G)$, then $\gamma_{\{R2\}}(G) \leq \gamma_{\{R2\}}(G\sim v)\leq \gamma_{\{R2\}}(G)+1$. Moreover, if there exists a $\gamma_{\{R2\}}(G)$-function
  $f$ such that $f(v)\neq 1$, then $\gamma_{\{R2\}}(G\sim v)=\gamma_{\{R2\}}(G)$.
\end{proposition}

\begin{proof} 
Let $h$ be a $\gamma_{\{R2\}}(G)$-function. The function  $g$ defined on $G\sim v$ by assigning $g(\tilde{v})=1$ and $g(u)=h(u)$ for every $u\in V(G)$ is  a Roman $\{2\}$-dominating function of $G\sim v$ with $g(G\sim v)=\gamma_{\{R2\}}(G)+1$, implying $\gamma_{\{R2\}}(G\sim v)\leq \gamma_{\{R2\}}(G)+1$.

For the lower bound, let $g$ be a $\gamma_{\{R2\}}(G\sim v)$-function. Note that $g(v)+g(\tilde{v}) \leq 2$, else $g$ is not minimum (the function $\overline{g}$ such that $\overline{g}(v)=\overline{g}(\tilde{v})=1$ and $\overline{g}(u)=g(u)$ for every other $u\in V(G)$ is a Roman $\{2\}$-dominating function of $G\sim v$ with $\overline{g}(G\sim v)<g(G\sim v)$). Then define $f$ over $V$ such that $f(v)=g(v)+g(\tilde{v})$ and $f(u)=g(u)$ for $u\in V$ and $u\neq v$. It is clear that $f$ is a Roman $\{2\}$-dominating function of $G$ with weight $\gamma_{\{R2\}}(G\sim v)$. Therefore, $\gamma_{\{R2\}}(G\sim v)\geq \gamma_{\{R2\}}(G)$.

Suppose now that there exists a $\gamma_{\{R2\}}(G)$-function $f$ such that $f(v)\neq 1$. If $f(v)=0$, since every neighbor of $v$ in $G$ is a neighbor of $\tilde{v}$, we can extend $f$ to a Roman $\{2\}$-dominating function of $G\sim v$ with weight $\gamma_{\{R2\}}(G)$.
If $f(v)=2$, we define a function $g$ by $g(v)=g(\tilde{v})=1$ and $g(u)=f(u)$ for every other $u\in V(G)$. Since every neighbour $w$ of $v$ with $g(w)=0$ is also a neighbour of $\tilde{v}$, $g(N(w))\geq 2$ and $g$ is a Roman $\{2\}$-dominating function of $G\sim v$ with $g(G)=\gamma_{\{R2\}}(G)$. Hence, $\gamma_{\{R2\}}(G\sim v)\leq \gamma_{\{R2\}}(G)$. 
\end{proof}

\begin{proposition}\label{prop:w12}
  Let $W$ be a well labelled spider obtained by performing operations R1 or R2 to a thin spider $G=(S,C,H)$. Then $\gamma_{\{R2\}}(W)= \gamma_{\{R2\}}(G)=|S|+1$.
\end{proposition}
\begin{proof}
On the one hand, notice that operation R1 is equivalent to performing on $G$ the composition of two additions of twins to a vertex $v\in S\cup C$. 
Then, from the proof of Proposition \ref{spider}, there exists a $\gamma_{\{R2\}}(G)$-function that assigns 0 to $v$. The result follows by applying Proposition \ref{twins} twice.

On the other hand, notice that performing operation R2 to $G$ is equivalent to adding a twin to both endpoints of a leg, let us say $cs$, of $G$. From the proof of Proposition $\ref{spider}$, there exists a $\gamma_{\{R2\}}(G)$-function $f=(V_0,V_1, V_2)$ of $G$ with $c\in V_2$ and $s\in V_0$. The result follows by applying Proposition \ref{twins}.\end{proof}

\begin{proposition}\label{prop:R3}
  Let $W$ be a well labelled spider  obtained by performing operation $R3$ on a thin spider $G=(S,C,H)$ onto $t\geq 1$ vertices in $S$. Then $\gamma_{\{R2\}}(W)= \gamma_{\{R2\}}(G)+t-1$.
\end{proposition}

\begin{proof}
Notice first that performing operation R3 to $G$ is equivalent to adding a twin to  vertices of $S$.  By applying iteratively Proposition \ref{twins}, the upper bound obtained is $\gamma_{\{R2\}}(G)+t$. We are going to prove that in fact, $\gamma_{\{R2\}}(W)\leq \gamma_{\{R2\}}(G)+t-1$.
  
For $j \in [1,|S|]$,  let $\tilde{s_j}$ be the twin ---added by operation R3--- of vertex $s_j$ in $S$. 
  
On the one hand, the function 
 $$f(v)=\left\lbrace \begin{array}{ccc}
   2 & &\text{ for } v=c_i \\
  1 & & \text{ for } v=s_j \text{ or } v=\tilde{s_j},\; j\neq i\\
  0 & &\text{otherwise,}
 \end{array}\right.$$
 is a Roman \{2\}-dominating function of $W$ with weight $f(W)=2+(|S|-1)+(t-1)=\gamma_{\{R2\}}(G)+(t-1)$. Then $\gamma_{\{R2\}}(W)\leq \gamma_{\{R2\}}(G)+(t-1)$ by Proposition \ref{spider}.
  
  On the other hand, let $g$ be a $\gamma_{\{R2\}}(W)$-function. Then $g(s_j)+g(c_j)+g(\tilde{s_j})=2$. Indeed, $g(s_j)+g(c_j)+g(\tilde{s_j})=1$ contradicts that $g$ is a Roman \{2\}-dominating function of $W$, and $g(s_j)+g(c_j)+g(\tilde{s_j})>2$ contradicts the fact that $g$ is minimum (the function $h$ such that $h(c_j)=2$, $h(s_j)=h(\tilde{s_j})=0$ and $h(u)=g(u)$ for every other vertex of $W$ is a Roman \{2\}-dominating function of $W$ with $h(W)<g(W)$).
  Then $$g(W)\geq \sum\limits_{i=1}^{|S|} g(s_i)+g(c_i)+g(\tilde{s_i}) =2|S|\geq |S|+t = \gamma_{\{R2\}}(G)+(t-1)$$ 
  by Proposition \ref{spider} and the fact that $t\leq |S|$. Therefore $\gamma_{\{R2\}}(W)\leq\gamma_{\{R2\}}(G)+(t-1)$.
  \end{proof}

\subsection{Graphs in $ZOO$}\label{zoo}

As in \cite{roussel1999} we call ZOO, the graph family containing the following graphs and their complements: paths $P_k$ with $k\geq 6$, cycles $C_k$ with $k\geq 5$, graphs $C^u_5$ (a cycle $C_5$ together with a vertex $u$ adjacent to at most 4 vertices of the $C_5$), graphs $P^u_5$ (a path $P_5$ together with a vertex $u$ adjacent to at most 4 vertices of the $P_5$), $J$, $K$, $L$ and $Q$ shown in Figure \ref{figJKLQ}, $J-\{w\}$, $J-\{v,w\}$ and the eleven graphs shown in Figure \ref{fig}.

\begin{figure*}[h]
\begin{center}
  \begin{tikzpicture}[scale=0.7]
    \begin{scope}[shift={(-7,0)}]

      \node (9) at (-0.2,0.6) {$J$};
      
      \foreach \x in {0,1,2,3,4,5}{
      \node[inner sep=0] (\x) at (\x*0.6,0) {}; }
      
       \foreach \x in {0,1,2,3,4}{
      \draw[thick] (\x) -- (\x*0.6+0.6,0);}

      \node[inner sep=0] (6) at (1.5,0.6) {}; 
      \node[inner sep=0] (7) at (2.1,0.6) {}; 
      \node[inner sep=0] (8) at (2.7,0.6) {};

       \foreach \x in {2,3}{
      \draw[thick] (\x) -- (6);}

      \draw[thick] (6) -- (7);
      \draw[thick] (7) -- (8);
    
      \foreach \x in {0,1,2,3,4,5,6,7,8}{
      \filldraw (\x) circle (1.6pt); }

      \node[inner sep=0] (6) at (1.6,1) {\footnotesize $u$}; 
      \node[inner sep=0] (7) at (2.2,1) {\footnotesize $v$}; 
      \node[inner sep=0] (8) at (2.8,1) {\footnotesize $w$};

   \end{scope}

   \begin{scope}[shift={(-2.5,0)}]
      \node (9) at (-0.2,0.6) {$K$};
      
      \foreach \x in {0,1,2,3,4,5}{
      \node[inner sep=0] (\x) at (\x*0.6,0) {}; }
      
       \foreach \x in {0,1,2,3,4}{
      \draw[thick] (\x) -- (\x*0.6+0.6,0);}

      \node[inner sep=0] (6) at (1.5,0.6) {}; 
      
       \foreach \x in {0,2,3}{
      \draw[thick] (\x) -- (6);}

      \foreach \x in {0,1,2,3,4,5,6}{
      \filldraw (\x) circle (1.6pt); }

      \node[inner sep=0] (6) at (1.9,1) {\footnotesize $u$};

   \end{scope}
   
   \begin{scope}[shift={(2,0)}]
      \node (9) at (-0.2,0.6) {$L$};
      
      \foreach \x in {0,1,2,3,4,5}{
      \node[inner sep=0] (\x) at (\x*0.6,0) {}; }
      
       \foreach \x in {0,1,2,3,4}{
      \draw[thick] (\x) -- (\x*0.6+0.6,0);}

      \node[inner sep=0] (6) at (1.5,0.6) {}; 
      
       \foreach \x in {1,2,3,4}{
      \draw[thick] (\x) -- (6);}

      \foreach \x in {0,1,2,3,4,5,6}{
      \filldraw (\x) circle (1.6pt); }

      \node[inner sep=0] (6) at (1.9,1) {\footnotesize $u$}; 
      \end{scope}

      \begin{scope}[shift={(6.5,0)}]
         \node () at (0,0.6){$Q$};
      
      \foreach \x in {0,1,2,3}{
      \node[inner sep=0pt] (\x) at (\x*0.8,0) {}; }
      
      \foreach \x in {0,1,2}{
      \draw[thick] (\x) -- ({\x*0.8+0.8},0);}

      \node[inner sep=0pt] (4) at (0.8,0.8) {};
      \node[inner sep=0pt] (5) at (1.6,0.8) {};
        
      \draw[thick] (1) -- (4);
      \draw[thick] (2) -- (5);
      \draw[thick] (4) -- (5);
    
      \foreach \x in {0,1,2,3,4,5}{
      \filldraw (\x) circle (1.6pt); }; 
      \end{scope}
  \end{tikzpicture}
  \caption{Graphs $J$, $K$, $L$ and $Q$}\label{figJKLQ}
\end{center}
\end{figure*}
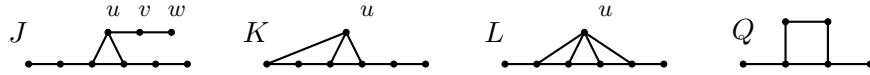
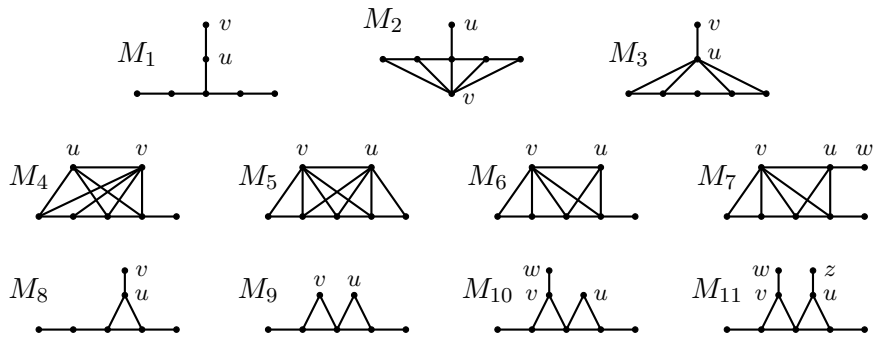
\begin{figure*}[h]
\begin{center}
    \begin{tikzpicture}[scale=0.65]
       \begin{scope}[shift={(-5,2.5)}]      
      \node (9) at (0,0.8){$M_1$};
      
      \foreach \x in {0,1,2,3,4}{
      \node[inner sep=0pt] (\x) at (\x*0.7,0) {}; }
      
       \foreach \x in {0,1,2,3}{
      \draw[thick] (\x) -- (\x*0.7+0.7,0);}

      \node[inner sep=0] (5) at (1.4,0.7) {}; 
      \node[inner sep=0] (6) at (1.4,1.4) {}; 
        
      \draw[thick] (2) -- (5);
      \draw[thick] (6) -- (5);
    
      \foreach \x in {0,1,2,3,4,5,6}{
      \filldraw (\x) circle (1.6pt); }

      \node[inner sep=0pt] (10) at (1.8,0.7) {\footnotesize $u$};
      \node[inner sep=0] (11) at (1.8,1.4) {\footnotesize $v$};     
    \end{scope}

     \begin{scope}[shift={(0,3.2)}]
      \node (9) at (0,0.8) {$M_2$};
      
      \foreach \x in {0,1,2,3,4}{
      \node[inner sep=0] (\x) at (\x*0.7,0) {}; }
      
       \foreach \x in {0,1,2,3}{
      \draw[thick] (\x) -- (\x*0.7+0.7,0);}

      \node[inner sep=0] (5) at (1.4,0.7) {}; 
      \node[inner sep=0] (6) at (1.4,-0.7) {}; 
      
      \foreach \x in {0,1,2,3,4}{
      \draw[thick] (\x) -- (6);}
      \draw[thick] (2) -- (5);
    
      \foreach \x in {0,1,2,3,4,5,6}{
      \filldraw (\x) circle (1.6pt); }

      \node[inner sep=0] (10) at (1.8,0.7) {\footnotesize $u$}; 
      \node[inner sep=0] (11) at (1.75,-0.8) {\footnotesize $v$}; 
    \end{scope}

    \begin{scope}[shift={(5,2.5)}]
      \node (9) at (0,0.8) {$M_3$};
      
      \foreach \x in {0,1,2,3,4}{
      \node[inner sep=0] (\x) at (\x*0.7,0) {}; }
      
       \foreach \x in {0,1,2,3}{
      \draw[thick] (\x) -- (\x*0.7+0.7,0);}

      \node[inner sep=0pt] (5) at (1.4,1.4) {}; 
      \node[inner sep=0pt] (6) at (1.4,0.7) {}; 
      
      \foreach \x in {0,1,3,4}{
      \draw[thick] (\x) -- (6);}
      
      \draw[thick] (6) -- (5);
    
      \foreach \x in {0,1,2,3,4,5,6}{
      \filldraw (\x) circle (1.6pt); }

      \node[inner sep=0pt] (10) at (1.75,1.4) {\footnotesize $v$}; 
      \node[inner sep=0pt] (11) at (1.75,0.8) {\footnotesize $u$}; 
      
   \end{scope}
   \begin{scope}[shift={(-7,0)}]
      \node (9) at (-0.2,0.8){$M_4$};
      
      \foreach \x in {0,1,2,3,4}{
      \node[inner sep=0] (\x) at (\x*0.7,0) {}; }
      
       \foreach \x in {0,1,2,3}{
      \draw[thick] (\x) -- (\x*0.7+0.7,0);}

      \node[inner sep=0] (5) at (0.7,1) {}; 
      \node[inner sep=0] (6) at (2.1,1) {}; 
      
      \foreach \x in {0,1,2,3}{
      \draw[thick] (\x) -- (6);}
      \foreach \x in {0,2,3}{
      \draw[thick] (\x) -- (5);}
      
      \draw[thick] (6) -- (5);
    
      \foreach \x in {0,1,2,3,4,5,6}{
      \filldraw (\x) circle (1.6pt); }

      \node[inner sep=0] (10) at (0.7,1.35) {\footnotesize $u$}; 
      \node[inner sep=0] (11) at (2.1,1.35) {\footnotesize $v$}; 
       
   \end{scope}
   
   \begin{scope}[shift={(-7/3,0)}]
      \node (9) at (-0.2,0.8) {$M_5$};
      
      \foreach \x in {0,1,2,3,4}{
      \node[inner sep=0] (\x) at (\x*0.7,0) {}; }
      
       \foreach \x in {0,1,2,3}{
      \draw[thick] (\x) -- (\x*0.7+0.7,0);}

      \node[inner sep=0] (5) at (0.7,1) {}; 
      \node[inner sep=0] (6) at (2.1,1) {}; 
      
      \foreach \x in {1,2,3,4}{
      \draw[thick] (\x) -- (6);}
      \foreach \x in {0,1,2,3}{
      \draw[thick] (\x) -- (5);}
      
      \draw[thick] (6) -- (5);
    
      \foreach \x in {0,1,2,3,4,5,6}{
      \filldraw (\x) circle (1.6pt); }

      \node[inner sep=0] (5) at (0.7,1.35) {\footnotesize $v$}; 
      \node[inner sep=0] (6) at (2.1,1.35) {\footnotesize $u$};

   \end{scope}
   \begin{scope}[shift={(7/3,0)}]
      \node (9) at (-0.2,0.8) {$M_6$};
      
      \foreach \x in {0,1,2,3,4}{
      \node[inner sep=0] (\x) at (\x*0.7,0) {}; }
      
       \foreach \x in {0,1,2,3}{
      \draw[thick] (\x) -- (\x*0.7+0.7,0);}

      \node[inner sep=0] (6) at (0.7,1) {}; 
      \node[inner sep=0] (5) at (2.1,1) {}; 
      
      \foreach \x in {2,3}{
      \draw[thick] (\x) -- (5);}
      \foreach \x in {0,1,2,3}{
      \draw[thick] (\x) -- (6);}
      
      \draw[thick] (6) -- (5);
    
      \foreach \x in {0,1,2,3,4,5,6}{
      \filldraw (\x) circle (1.6pt); }

      \node[inner sep=0] (5) at (0.7,1.35) {\footnotesize $v$}; 
      \node[inner sep=0] (6) at (2.1,1.35) {\footnotesize $u$};

   \end{scope}
   \begin{scope}[shift={(7,0)}]
      \node (9) at (-0.2,0.8) {$M_7$};
      
      \foreach \x in {0,1,2,3,4}{
      \node[inner sep=0] (\x) at (\x*0.7,0) {}; }
      
       \foreach \x in {0,1,2,3}{
      \draw[thick] (\x) -- (\x*0.7+0.7,0);}

      \node[inner sep=0] (6) at (0.7,1) {}; 
      \node[inner sep=0] (5) at (2.1,1) {}; 
      \node[inner sep=0] (7) at (2.8,1) {}; 
      \foreach \x in {2,3}{
      \draw[thick] (\x) -- (5);}
      \foreach \x in {0,1,2,3}{
      \draw[thick] (\x) -- (6);}
      
      \draw[thick] (6) -- (5);
      \draw[thick] (7) -- (5);
    
      \foreach \x in {0,1,2,3,4,5,6,7}{
      \filldraw (\x) circle (1.6pt); }

      \node[inner sep=0] (5) at (0.7,1.35) {\footnotesize $v$}; 
      \node[inner sep=0] (6) at (2.1,1.35) {\footnotesize $u$}; 
      \node[inner sep=0] (7) at (2.8,1.35) {\footnotesize $w$}; 
       
   \end{scope}  
   
   \begin{scope}[shift={(-7,-2.3)}]
      \node (9) at (-0.2,0.8) {$M_8$};
      
      \foreach \x in {0,1,2,3,4}{
      \node[inner sep=0] (\x) at (\x*0.7,0) {}; }
      
       \foreach \x in {0,1,2,3}{
      \draw[thick] (\x) -- (\x*0.7+0.7,0);}

      \node[inner sep=0] (5) at (1.75,0.7) {}; 
      \node[inner sep=0] (6) at (1.75,1.2) {}; 

      \foreach \x in {2,3}{
      \draw[thick] (\x) -- (5);}
      \draw[thick] (6) -- (5);

      \foreach \x in {0,1,2,3,4,5,6}{
      \filldraw (\x) circle (1.6pt); }

      \node[inner sep=0] (5) at (2.1,0.7) {\footnotesize $u$}; 
      \node[inner sep=0] (6) at (2.1,1.2) {\footnotesize $v$};
      
    \end{scope}
   
   \begin{scope}[shift={(-7/3,-2.3)}]
      \node (9) at (-0.2,0.8) {$M_9$};
      
      \foreach \x in {0,1,2,3,4}{
      \node[inner sep=0] (\x) at (\x*0.7,0) {}; }
      
       \foreach \x in {0,1,2,3}{
      \draw[thick] (\x) -- (\x*0.7+0.7,0);}

      \node[inner sep=0] (5) at (1.05,0.7) {}; 
      \node[inner sep=0] (6) at (1.75,0.7) {}; 
      
       \foreach \x in {1,2}{
      \draw[thick] (\x) -- (5);}
       \foreach \x in {2,3}{
      \draw[thick] (\x) -- (6);}
    
      \foreach \x in {0,1,2,3,4,5,6}{
      \filldraw (\x) circle (1.6pt); }

      \node[inner sep=0] (5) at (1.05,1) {\footnotesize $v$}; 
      \node[inner sep=0] (6) at (1.75,1) {\footnotesize $u$}; 
      
    \end{scope}
   
     \begin{scope}[shift={(7/3,-2.3)}]
      \node (9) at (-0.2,0.8) {$M_{10}$};
      
      \foreach \x in {0,1,2,3,4}{
      \node[inner sep=0] (\x) at (\x*0.7,0) {}; }
      
       \foreach \x in {0,1,2,3}{
      \draw[thick] (\x) -- (\x*0.7+0.7,0);}
                
      \node[inner sep=0] (5) at (1.05,0.7) {}; 
      \node[inner sep=0] (6) at (1.75,0.7) {}; 
      \node[inner sep=0] (7) at (1.05,1.2) {}; 
      
       \foreach \x in {1,2}{
      \draw[thick] (\x) -- (5);}
       \foreach \x in {2,3}{
      \draw[thick] (\x) -- (6);}

      \draw[thick] (5) -- (7);
    
      \foreach \x in {0,1,2,3,4,5,6,7}{
      \filldraw (\x) circle (1.6pt); }

      \node[inner sep=0] (5) at (0.7,0.7) {\footnotesize $v$}; 
      \node[inner sep=0] (6) at (2.1,0.7) {\footnotesize $u$}; 
      \node[inner sep=0] (7) at (0.7,1.2) {\footnotesize $w$}; 
      
     \end{scope}
   
       \begin{scope}[shift={(7,-2.3)}]
      \node (9) at (-0.2,0.8) {$M_{11}$};
      
      \foreach \x in {0,1,2,3,4}{
      \node[inner sep=0] (\x) at (\x*0.7,0) {}; }
      
      \foreach \x in {0,1,2,3}{
      \draw[thick] (\x) -- (\x*0.7+0.7,0);}

      \node[inner sep=0] (5) at (1.05,0.7) {}; 
      \node[inner sep=0] (6) at (1.75,0.7) {}; 
      \node[inner sep=0] (7) at (1.05,1.2) {}; 
      \node[inner sep=0] (8) at (1.75,1.2) {}; 
      
       \foreach \x in {1,2}{
      \draw[thick] (\x) -- (5);}
       \foreach \x in {2,3}{
      \draw[thick] (\x) -- (6);}

      \draw[thick] (5) -- (7);
      \draw[thick] (6) -- (8);
    
      \foreach \x in {0,1,2,3,4,5,6,7,8}{
      \filldraw (\x) circle (1.6pt); }

      \node[inner sep=0] (5) at (0.7,0.7) {\footnotesize $v$}; 
      \node[inner sep=0] (6) at (2.1,0.7) {\footnotesize $u$}; 
      \node[inner sep=0] (7) at (0.7,1.2) {\footnotesize $w$}; 
      \node[inner sep=0] (6) at (2.1,1.2) {\footnotesize $z$}; 
      \end{scope} 

  \end{tikzpicture}
\caption{Graphs $M_i$ in ZOO, for $i \in [1, 11]$}
  \label{fig}
\end{center}
\end{figure*}

In this subsection we exhibit the Roman $\{2\}$-domination number for each graph in ZOO, except for paths and cycles whose Roman $\{2\}$-domination numbers are already known from \cite{Chellali2016}. For an $n$-cycle $C_n$, $\gamma_{\{R2\}}(C_n)=\left\lceil\frac{n}{2}\right\rceil$ and, for an $n$-path $P_n$, $\gamma_{\{R2\}}(P_n)=\left\lceil\frac{n+1}{2}\right\rceil$\cite{Chellali2016}.
We omit the proofs  because of their simplicity.

\begin{proposition}\label{prop8}
$\gamma_{\{R2\}}(C_5^u)=3$ if $u$ is adjacent to at least two vertices of $C_5$. Else, $\gamma_{\{R2\}}(C^u_5)=4$. 
\end{proposition}

Note that $\overline{C_5}=C_5$ and $\overline{C_5^u}$ is also a graph $C_5^u$. For the remaining complements of cycles in ZOO we prove first the following result:

\begin{proposition}\label{complement}
  Let $G$ be a connected graph with $|V|\geq 3$ and $\delta(G)=1$. Then $\gamma_{\{R2\}}(\overline{G})=3$.
\end{proposition}

\begin{proof}
Let $u\in V$ with $deg(u)=1$ and let $v$ be the only neighbor of $u$. Then $f=(V\setminus\{u,v\},\{v\},\{u\})$ is a Roman $\{2\}$-dominating function of $\overline{G}$ since the only non-neighbor of $u$ in $\overline{G}$ is $v$. Thus, $\gamma_{\{R2\}}(\overline{G})\leq 3$.

Let us assume that $\gamma_{\{R2\}}(\overline{G})=2$ and let $g$ be a $\gamma_{\{R2\}}(\overline{G})$-function. If $g(x)=2$ for some $x\in V$, then $x$ is a universal vertex in $\overline{G}$ (i.e. an isolated vertex in $G$), which contradicts the fact that $G$ is connected. If $g(x)=g(y)=1$ for some $y$ distinct from $x$, then every other vertex is adjacent in $\overline{G}$ to both $x$ and $y$ and $x$ and $y$ are non-neighbors in $\overline{G}$ as otherwise, $x$ (and also $y$) would be a universal vertex in $\overline{G}$ contradicting the fact that $G$ is connected. Therefore, $\gamma_{\{R2\}}(\overline{G})\geq 3$. 
\end{proof}

\begin{proposition}\label{prop9}
$\gamma_{\{R2\}}(\overline{C_k})=3$ for $k> 5$.
\end{proposition}

\begin{proposition}\label{prop10}
$\gamma_{\{R2\}}(P_5^u)=3$ if $u$ is adjacent to at least two vertices $v\in V(P_5)$ with $deg_{P_5}(v)=1$, or to exactly one of them and also to the central vertex in $P_5$. In any other case, $\gamma_{\{R2\}}(P_5^u)=4$.

\end{proposition}

\begin{proposition}\label{Jvw}
 $\gamma_{\{R2\}}(J-\{v,w\})=\gamma_{\{R2\}}(Q)=\gamma_{\{R2\}}(K)=\gamma_{\{R2\}}(L)=4$, $\gamma_{\{R2\}}(J-\{w\})=5$ and $\gamma_{\{R2\}}(J)=6$, where graphs $J$, $K$, $L$ and $Q$ are shown in Figure \ref{figJKLQ}.
\end{proposition}

\begin{proposition}\label{Mis}
  Let $M_i$ with $i\in [1, 11]$ be the graphs shown in Figure \ref{fig}. Then
  $$\gamma_{\{R2\}} (M_i)=\left\lbrace\begin{array}{cl}
3, & \text{ for } i\in [2,6], \\
4, & \text{ for } i=1 \text{ or } i\in [7,9], \\
5, & \text{ for } i\in [10,11].
 \end{array}\right.$$
\end{proposition}

\medskip

\begin{proposition}
If $G$ is one of $P_k$ for $k\geq 6$, $J$, $K$, $L$, $Q$, $M_i$ with $i\neq 5$ or $P_5^u$ with $\delta(P_5^u)= 1$, 
then $\gamma_{\{R2\}}(\overline{G})=3$.  
\end{proposition}

\subsection{ Prime split $\{H_1,H_2,\overline{H_1},\overline{H_2}\}$-free graphs}

In this section, we make use of the characterization given in \cite{roussel1999} (stated in Lemma \ref{split} below) of the subclass of prime split graphs that have neither $H_1$, $H_2$ nor their complements as induced subgraphs (see below Figure \ref{Hs}).

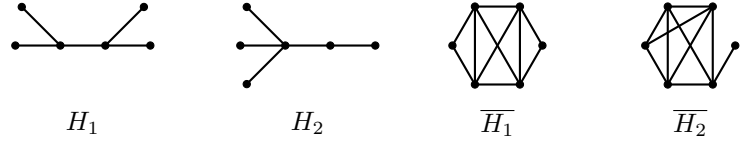
\begin{figure*}[h]
\centering
\begin{tikzpicture}[scale=0.85]
  \begin{scope}[shift={(-5.5,0)}]
    \node () at (1.05,-1.2){\small $H_1$};
      
      \foreach \x in {0,1,2,3}{
      \node[inner sep=0pt] (\x) at (\x*0.7,0) {}; }
      
       \foreach \x in {0,1,2}{
      \draw[thick] (\x) -- (\x*0.7+0.7,0);}

      \node[inner sep=0] (4) at (0.1,0.6) {}; 
      \node[inner sep=0] (5) at (2,0.6) {}; 
        
      \draw[thick] (1) -- (4);
      \draw[thick] (5) -- (2);
    
      \foreach \x in {0,1,2,3,4,5}{
      \filldraw (\x) circle (1.6pt); }
   \end{scope}
  \begin{scope}[shift={(-2,0)}]
    \node () at (1.05,-1.2){\small $H_2$};
      
      \foreach \x in {0,1,2,3}{
      \node[inner sep=0pt] (\x) at (\x*0.7,0) {}; }
      
       \foreach \x in {0,1,2}{
      \draw[thick] (\x) -- (\x*0.7+0.7,0);}

      \node[inner sep=0] (4) at (0.1,0.6) {}; 
      \node[inner sep=0] (5) at (0.1,-0.6) {}; 
        
      \draw[thick] (1) -- (4);
      \draw[thick] (5) -- (1);
    
      \foreach \x in {0,1,2,3,4,5}{
      \filldraw (\x) circle (1.6pt); }
    \end{scope} 
    \begin{scope}[shift={(2,0)}]
       \node () at (0,-1.2){\small $\overline{H_1}$};
    \foreach \x in {0,1,2,3,4,5}{
      \node[inner sep=0pt] (\x) at (\x*60:0.7) {}; 
      \draw[thick] (\x) -- (\x*60+60:0.7);
      \filldraw (\x) circle (1.6pt);}

      \draw[thick] (1) -- (4);
      \draw[thick] (1) -- (5);
      \draw[thick] (2) -- (4);
      \draw[thick] (2) -- (5);
       
    \end{scope}
    \begin{scope}[shift={(5,0)}]
       \node () at (0,-1.2){\small $\overline{H_2}$};
    \foreach \x in {0,1,2,3,4,5}{
      \node[inner sep=0pt] (\x) at (\x*60:0.7) {}; 
      \filldraw (\x) circle (1.6pt);}
       \foreach \x in {1,2,3,4,5}{
      \draw[thick] (\x) -- (\x*60+60:0.7);}

      \draw[thick] (1) -- (3);
      \draw[thick] (1) -- (4);
      \draw[thick] (1) -- (5);
      \draw[thick] (2) -- (4);
      \draw[thick] (2) -- (5);
       
    \end{scope}
     
\end{tikzpicture}
\caption{Graphs $H_1$, $H_2$ and their complements}\label{Hs}
\end{figure*}

For a graph $G=(V,E)$, a set of vertices $M\subseteq V$ is a \emph{module} if every vertex in $V\setminus M$ is either adjacent to every vertex in $M$, or to none of them. A graph with at least 2 vertices is \emph{prime} is it only has trivial modules  ($\emptyset,\, V$ or a singleton).

Let $G=(S, K )$ be a split graph, where $S$ is a stable set, and $K$ is a clique. The following notation and result (Lemma \ref{split}) are presented in \cite{roussel1999}. Let $\ell$ be the maximum degree in $G$ of the vertices of $S$ and, for any $i\in [1, \ell]$ denote

$A_i=\{s\in S\;\colon \; deg_G(s)=i\}$, \;
$B_i=\bigcap\limits_{s\in A_i} N(s)$, \; $C_i=\bigcup\limits_{s\in A_i} N(s)$ \; and \; 
$C_i'=\bigcup\limits_{j=1}^{i-2} C_j$.

Also, consider the following definitions:
$A_i$ is a \emph{thin-stable} if $|B_i|\geq i-1$ and a \emph{thick-stable} if $|C_i|\leq i+1$.

\begin{remark}[\cite{roussel1999}]\label{remark 1}
  For each $i \in [1, \ell]$, notice that $A_i$ is a thick-stable or a thin-stable if and only if for any $x,y \in A_i$ with $x\neq y$, $|N(x)\cap N(y)|=i-1$. 

\end{remark}

A generalization of Remark \ref{remark 1} will be useful for the upcoming results: 

\begin{remark}\label{general}
Let $i \in [1, \ell]$ and  $A_i=\{ s_1, s_2,\ldots,s_{|A_i|}\}$. $A_i$ is a thick-stable or a thin-stable if and only if for any $|\bigcap_{j=1}^{k+1} N(s_j)|=i-k$, for each $ k \in [1, i-1]$.
\end{remark}

\begin{lemma}[\cite{roussel1999}]\label{split}
  The prime split graph $G=(S, K)$ is $\{H_1,H_2,\overline{H_1},\overline{H_2}\}$-free if and only if for all $i\in [1, \ell]$, the following statements hold:
  \begin{enumerate}
    \item $A_i$ is a thin-stable or a thick-stable.
    \item $C_i'\subseteq B_i$.
    \item If $|A_i|\geq 4$ and $A_i$ is a thick-stable then $i=\ell$. Moreover, if $A_{i-1}\neq \emptyset$ then $|A_{i-1}|\leq 3$, $C_{i-1}\subseteq C_i$ and $A_i\cup A_{i-1}\cup K$ is $\{H_1,H_2\}$-free.
    \item If $|A_{i-1}|\geq 4$ and $A_{i-1}$ is a thin-stable then $i=2$, $B_{i-1}\subseteq B_i$ and $A_i\cup A_{i-1}\cup K$ is $\{H_1,H_2\}$-free.
    \item If $|A_i|<4$ and $|A_{i-1}|<4$ then for all $s\in A_i$, $s$ does not belong to a subgraph of $A_i\cup A_{i-1}\cup S_i\cup K$ isomorphic to $H_1$ or $H_2$, and for all $s' \in A_{i-1}$, $|N(s')\setminus N(s)|\leq 1$.
  \end{enumerate}
\end{lemma}

From Lemma \ref{split}, we derive the following assertions.

\begin{corollary}\label{remark 2} From Lemma \ref{split} holds:
  
\begin{itemize}
  \item[i.] If $|A_i|\geq 4$ then $i=1$ or $i=\ell$, and then $|A_i|\leq 3$ for $i \in [2,\ell-1]$.
  \item[ii.] $B_j\subseteq C_j\subseteq B_i$ for $j \in [1, i-2]$ and $i \in [3, \ell]$.
  \item[iii.] If $B_j=\emptyset$ for some $j\geq 3$, then $C_j'=\emptyset$ and $A_i\subseteq C_i=\emptyset$ for each $i\in [1, j-2]$. 
  \item[iv.] For $i \in [2, \ell]$, $B_i = \emptyset$ if and only if $|A_i|=i+1$ or $A_i= \emptyset$. Even more, if $|A_i|=i+1$ then every vertex in $C_i$ is adjacent to exactly $i$ vertices in $A_i$.
  \item[v.] If $|A_i|<4$, $|A_{i+1}|<4$ and \begin{itemize}
    \item[-] $C_i\cap B_{i+1}=\emptyset$ for $i=2,3$ or
    \item[-] $B_i\cap B_{i+1}=\emptyset$ for $i \in [4, \ell-1]$,
  \end{itemize}
  then $C_i\subseteq C_{i+1}$ and every vertex in $B_i$ is not adjacent to at most one vertex in $A_{i+1}$.
  \item[vi.] If $\ell>2$, then $A_1$, $A_2$ and $A_{\ell}$ are the only non-empty sets among the $A_i$'s with $B_i=\emptyset$.
\end{itemize}
\end{corollary}

The previous results allow to study the domination and Roman $\{2\}$-domination numbers of prime split $\{H_1,H_2,\overline{H_1},\overline{H_2}\}$-free graphs. 

\begin{theorem}\label{gamasplit}
    Let $G=(S, K)$ be a prime split $\{H_1,H_2,\overline{H_1},\overline{H_2}\}$-free graph, $\ell$ and the sets $A_i$, $B_i$ and $C_i$ for $i \in [1,\ell]$ be defined as above, and $i_{min}=\min\limits_{B_i\neq\emptyset}\{i¨\in [1,\ell] \}$. Then $|A_1| \leq \gamma (G) \leq |A_1|+2$.
\end{theorem}
\begin{proof}
From Lemma \ref{split} $A_1$ is a subset of pendant vertices whose neigbourhoods are pairwise disjoint, then $\gamma(G)\geq |A_1|$. 

From ii and iv in Corollary \ref{remark 2}, if $A_2=\emptyset$ then $C_1\subseteq B_j$ for every $j\geq 3$, and thus $C_1$ is a dominating set in $G$. In this case, we have  $\gamma\leq |C_1|=|A_1|$.  If not ($A_2\neq\emptyset $), from Corollary \ref{general} there are two vertices in $C_2$, let's say $c_1$ and $c_2$, such that every vertex in $A_2$ is adjacent to either one of them. Then, $C_1\cup \{c_1,c_2\}$ is a dominating set in $G$ and thus $\gamma (G) \leq |C_1|+2=|A_1|+2$.

Now, let us assume $C_1=\emptyset$ (i.e. $A_1=\emptyset)$. We will divide the study into the different possible values  of $i_{min}$. 

If $i_{min}\neq 3$ or $i_{min}= \ell$, then $C_2\subseteq B_{i_{min}}$ and  the previous $\{c_1,c_2\}\subseteq C_2$ is a dominating set in $G$.

If  $i_{min}= 3$, clearly if $C_2\cap B_3\neq \emptyset$. Otherwise, from v. in  Remark \ref{general} $C_2\subseteq C_3$ and from $5.$ in Lemma \ref{split}, there is a subset of  $C_2$ with two vertices that again is a dominating set in $G$.

If $i_{min}= \ell-1$, $|A_{i_{min}}|<4$ and $|A_{i_{min}+1}|<4$. Then if $B_{i_{min}}\cap B_{i_{min}+1}\neq \emptyset$ there is a universal vertex in $G$. Otherwise, $C_{i_{min}}\subseteq  C_{i_{min}+1}$ and then, from v. in Remark  \ref{general}, we can build a dominating set in $G$ by considering a vertex in $B_{i_{min}}$ along with a neighbour from the non adjacent vertex in $A_{i_{min}+1}$.

Finally, if $B_i=\emptyset$ for every $i\in [1 \ell]$, then $C_i=\emptyset$ for every $i\in [1, \ell-1]$. Then there exist two vertices in $C_\ell$ forming a dominating set in  $G$.

In all the cases with $C_1=\emptyset$, we have built a dominating set in $G$ of size 2, implying $\gamma(G)\leq 2$.
\end{proof}

\begin{theorem}\label{gamaR2split}
   Let $G=(S, K)$ be a prime split \newline $\{H_1,H_2,\overline{H_1},\overline{H_2}\}$-free graph, $\ell$ and the sets $A_i$, $B_i$ and $C_i$ for $i \in [1,\ell]$ be defined as above, and $i_{min}=\min\limits_{B_i\neq\emptyset}\{i¨\in [1, \ell] \}$. Then   
$|A_1|+1 \leq \gamma_{\{R2\}} (G) \leq |A_1|+3$.
\end{theorem}

\begin{proof}
  We consider the different possible values of $\ell$.  
  If $\ell = 1 $, from Remark \ref{remark 1} follows that $G$ is a thin spider with partition $(A_1, C_1, \emptyset)$. By applying Proposition \ref{spider}, we obtain 
$\gamma_{\{R2\}}(G)=|A_1|+1$. If $\ell \geq 2$, from Proposition \ref{stem} and taking into account Remark \ref{remark 1}, follows $\gamma_{\{R2\}}(G)\geq |A_1|+1$. 

 The analysis for the upper bound will be separeted into to three cases:
 
  {\noindent \textbf{Case $\ell$=2: }}    
    From v. in Corollary \ref{remark 2} and Remark \ref{general}, if $B_2=\emptyset$ then $|A_2|=3$ and every vertex in $C_2$ is adjacent to two vertices in $A_2$. Assigning 1 to every vertex in $A_1$ and to two vertices in $C_2$, and zero to every other vertex we get a \r2d function of $G$. Else, if $B_2\neq \emptyset$, assigning 1 to every vertex in $A_1$, 2 to a vertex in $B_2$ and zero to every other vertex we have a \r2d function. In both cases the weight of the function is $|A_1|+2$.
    
  {\noindent  \textbf{Case $\ell$=3:} } 
  
   \emph{Subcase $B_3\neq \emptyset$}: $C_1\subseteq B_3$ from ii in Corollary \ref{remark 2}. If $A_2=\emptyset$, we assign 1 to every vertex in $A_1$ but one, and two to the neighbour of such vertex to get the desired bound. 
   If  $A_2 \neq \emptyset$, and $B_2\neq \emptyset$ assigning 2 to a vertex in $B_2$, and the previous function to every other vertex we get the desired bound. Now let us assume that $B_2=\emptyset$, then $|A_2|=3$ from v. in Corollary \ref{remark 2}. 
   
    If $|A_3|\geq 4$, then $A_3$ is not a thin-stable from item $4$ in Lemma \ref{split} and $|B_3|\leq 1$. But from Remark \ref{remark 1}, $B_3=\emptyset$, which leads to a contradiction.  

   If $|A_3|<4$: from item $5$ in Lemma \ref{split}, $N(s')\cap N(s)\neq \emptyset$ for every $s \in A_3$ and $s'\in A_2$. 
    Then assigning 1 to every vertex in $C_2\cup A_1$ and 0 to every other vertex, we get a $\{2\}$-Roman function of $G$ with weight $|A_1|+3$.  
         
  \emph{Subcase $B_3=\emptyset$}: then $|A_3|=|C_3|=4$ from iv in Corollary \ref{remark 2}, 
  $A_1=\emptyset$ from iii. in Corollary \ref{remark 2} and $C_2\subseteq C_3$ from item $3$ in Lemma \ref{split}. Since $B_3=\emptyset$, there are no two nonadjacent vertices of degree $n-2$. 
   If $A_2\neq \emptyset$ and $B_2=\emptyset$ then $|C_2|=3$ from iv. in Corollary \ref{remark 2} and Remark \ref{remark 1}. By assigning 1 to the vertices in $C_2$ and 0 to every other vertex in $G$, we get a Roman \{2\}-dominating function of $G$, because every vertex in $A_3$ is not adjacent to at most one vertex in $C_2$ ($|C_3\setminus C_2|=1$). Thus, $\gamma_{\{R2\}}(G)\leq|A_1|+3=3$. 
    Otherwise ($B_2\neq \emptyset$), we obtain a Roman \{2\}-dominating function of $G$ by assigning 2 to one (any) vertex in $B_2$, 1 to the vertex in $A_3$ not adjacent to it, and 0 to the remaining vertices. Then, $\gamma_{\{R2\}}(G)\leq|A_1|+3=3$.

  If $A_2=\emptyset$, by assigning 2 to one (any) vertex in $C_3$, 1 to the vertex in $A_3$ not adjacent to it, and 0 to the remaining vertices, we obtain a Roman \{2\}-dominating function of $G$. Thus, $\gamma_{\{R2\}}(G)\leq |A_1|+3=3$.

  {\noindent  \textbf{Case $\ell \geq$ 4:}}
According to iii. in Corollary \ref{remark 2}, since $B_i\subset B_{i+2}$ but possibly $B_i\nsubseteq B_{i+1}$, let us consider the different possibilities for $i_{min}=\min\limits_{B_i\neq\emptyset}\{i¨\in [1, \ell] \} $:
   
   \emph{Subcase $i_{min}\leq 3$}: clearly if $B_1 \neq\emptyset$ then $|A_1|=1$ and $B_1\subseteq B_j$ for every $j\geq 3$. We then proceed analogously to the case $\ell\leq 3$.
  
  \emph{Subcase $4\leq i_{min} \leq \ell-1$}: in this case $A_1= \emptyset$ because $C_1\subseteq B_3=\emptyset$ from item $2$ in Lemma $\ref{split}$. Also, due to vi. in Corollary \ref{remark 2}, $A_i=\emptyset$ for every $i\leq k-1$ but possibly $A_2$. 
  From ii. in Corollary \ref{remark 2}, $C_2\subseteq B_{i_{min}}\cap B_{i_{min}+1}$.

   \begin{itemize}
   \item 
 If $B_{i_{min}}\cap B_{i_{min}+1}=\emptyset$,  let $u$ be a vertex in $B_{i_{min}}$. Then $u\in B_j$ for $j\in [i_{min}, \ell]$ with  $j\neq i_{min}+1$. But also from v. in Corollary \ref{remark 2}, $u$ is not adjacent to at most one vertex in $A_{{i_{min}}+1}$. Then by assigning 2 to $u$, 1 to the vertex in $A_{{i_{min}}+1}$ that is not adjacent to $u$, and 0 to the remaining vertices, we build a Roman $\{2\}$-dominating function of $G$ with weight 3, and then $\gamma_{\{R2\}}(G)\leq 3$.

 \item If $B_{i_{min}}\cap B_{i_{min}+1}\neq \emptyset$ and $A_2\neq \emptyset$, the proof is analogous to the previous case. By assigning 2 to a vertex $u$ in $B_{i_{min}}$, 1 to the vertex in $A_2$ that is not not adjacent to $u$, and 0 to the remaining vertices, we get a Roman $\{2\}$-dominating function of $G$ arriving to $\gamma_{\{R2\}}(G)\leq 3$.

 \item If $B_{i_{min}}\cap B_{i_{min}+1}\neq \emptyset$ and $A_2= \emptyset$, then a vertex in the intersection ($B_{i_{min}}\cap B_{i_{min}+1}$) is a vertex in $B_j$ for every $j\geq i_{min}$ and therefore, a universal vertex. Then $\gamma_{\{R2\}}(G)=2$.
    \end{itemize}

  \emph{Subcase $i_{min}=\ell$}: from iii. in Corollary \ref{remark 2}, $|A_j|=\emptyset$ for $1<j<\ell$, since $B_j=\emptyset$ and $|A_j|\leq 4$. Besides $A_1=\emptyset$ because $A_3=\emptyset$. Then, a vertex in $B_j$ is a universal vertex and $\gamma_{\{R2\}}(G)=2$.
 
  \emph{Subcase $B_i=\emptyset$ for every $i$}:, from iv in Corollary \ref{remark 2} we have $|A_\ell|=\ell+1\geq 4$. 
  Let $s'\in A_\ell$ and consider $A_\ell\setminus \{s'\}$ and $B_\ell'=\bigcap\limits_{s\in A_\ell\setminus \{s'\}}N(s)$. Note that $B_\ell' \neq \emptyset$ from Remark \ref{remark 1}. Then by assigning 2 to some vertex in $B_\ell'$ and 1 to $s'$, we obtain a Roman $\{2\}$-dominating function of $G$ with weight 3, implying $\gamma_{\{R2\}}(G)\leq 3 $. 
    \end{proof}

\section{$\{2\}$-Roman  graph recognition of non decomposable partner limited graphs }\label{sec:sec5}

In the final section of this work, we characterize  the $\{2\}$-Roman property among the non decomposable graphs addressed in the previous section.

\begin{theorem}\label{th:special}
Among the graph classes studied in Section \ref{sec:special},  the only which are $\{2\}$-Roman graphs are: $P_6$, $C_5^u$ with $d(u)=1$, some particular $P_5^u$ and those well labelled spiders obtained by applying operation $R_3$ over every vertex in $S$, graphs $L$, $Q$ and $M_9$
  and  split $\{H_1,H_2,\overline{H_1},\overline{H_2}\}$-free graphs with a universal vertex, or $|A_1|=1$, $|A_2|=2$, $B_2\cap C_1=\emptyset$ and $\ell\geq 3$.
    
\end{theorem}

\begin{proof}
A $\{2\}$-Roman graph $G$ satisfies $\gamma_{\{R2\}}(G)=2\gamma(G)$, thus clearly those graphs for which $\gamma_{\{R2\}}(G)$ is odd are not $\{2\}$-Roman graphs. 

From Proposition  \ref{spider}, for a spider $(S,C,H)$, the $\{2\}$-Roman domination number is equal to $|S|+1$ for a thin spider and equal  to 3 for a thick spider. Besides, if $G$ is a thin spider, $\gamma(G)=|S|$ (since it has $|S|$ pendant vertices). Thus, any  spider is not a $\{2\}$-Roman graph since $|S|\geq 2$ from definition.

For  a well labelled spider $W$ obtained from a thin spider $(S,C,H)$, when $R_1$ or $R_2$ is applied we have $\gamma_{\{R2\}}(W)=|S|+1$ from Proposition \ref{prop:w12}, and $\gamma_{\{R2\}}(W)=|S|+t$ from Proposition \ref{prop:R3} when $R_3$ is applied over $t$ vertices of $S$.
In any case, $C$ is a minimum  dominating set in $W$,  thus $\gamma(W)=|S|$. This implies that $W$ is a $\{2\}$-Roman graph  only if operation $R_3$ is applied over each vertex  of $S$. 

For the graphs in $ZOO$ which $\gamma_{\{R2\}}(G)$ is even, it is not hard to verify that $C_5^u$ with $d(u)=1$,  $L$, $Q$ and $M_9$, and also some particular $P_5^u$s, are the only $\{2\}$-Roman graphs. 

    For a split $\{H_1,H_2,\overline{H_1},\overline{H_2}\}$-free graph $G$, since $|A_1|+1\leq \gamma_{\{R2\}}(G)\leq |A_1|+3$ and $|A_1|\leq \gamma(G) \leq |A_1|+2$, it is clear that $G$ is a $\{2\}$-Roman graph if and only if any of the following conditions holds:
    \begin{enumerate}
        \item $|A_1|=1$, $\gamma(G)=1$ and $\gamma_{\{R2\}}(G)=2$,
        \item $A_1=\emptyset$, $\gamma(G)=1$ and $\gamma_{\{R2\}}(G)=2$ or 
        \item $|A_1|=1$, $\gamma(G)=|A_1|+1=2$ and $\gamma_{\{R2\}}(G)=|A_1|+3=4$.
       \end{enumerate}
    Notice that the first two cases imply that $G$ has a universal vertex. 
    
    For the third case, let $D=\{u,v\}$ be a dominating set in $G$. Since every vertex in $A_1$ or its neighbour (in $C_1$) belongs to $D$, for optimality purpose we  consider $C_1\subseteq D$, let us say $C_1=\{u\}$. Then, every vertex not adjacent to $u$ belongs to $A_2$ and $v\in B_2$. If there is only one vertex not adjacent to $u$, clearly $\gamma_{\{R2\}}(G)\leq 3$. The same happens if $A_3=\emptyset$ (and then $|A_2|= 2$).
    Regarding Roman $\{2\}$-domination, notice that if $\ell\geq 3$, a weight of at least 2 is needed for  the vertices in $A_1\cup A_3$ to be dominated. In fact, since  $C_1\subseteq B_3$, we have $A_1\cup A_3 \subseteq S$. Besides, because $C_1\cap C_2$ and $B_3\cap C_2$ could be both empty, a weight 2 is needed for  the vertices in $A_2$. Also, since $B_2\cap C_1 = \emptyset$, there are no three vertices in $K$ such that every other vertex is adjacent to two of them. Therefore $\gamma_{\{R2\}}(G)=4$ and $G$ is a $\{2\}$-Roman graph. 
\end{proof}

\begin{corollary}
    Given a non decomposable partner limited graph $G$,  it is  polynomial time recognizable if $G$ is $\{2\}$-Roman, i.e. if $G$ is one of the graphs covered by Theorem \ref{th:special}. 
\end{corollary}

\begin{proof}
Let $G$ be a non decomposable partner-limited graph in the sense that both $G$ and $\overline{G}$ are connected. Then $G$ is a spider, a well labelled spider, a graph in the ZOO or else, a prime split $\{H_1,H_2,\overline{H_1},\overline{H_2}\}$-free graph. Since graphs in ZOO are cycles or paths or else have at most nine vertices, the conditions of Theorem \ref{th:special} for $G$ to be $\{2\}$-Roman can be all checked clearly in polynomial time. If $G$ is not any of the  graphs in ZOO, we can check in linear time if it is split or not \cite{Hammer1981}. If $G$ is not a split graph, it has to be a well labelled spider. Notice that, if $G$ is a $\{2\}$-Roman graph that is also a well labelled spider obtained from a spider  $(S,C,K)$, then every vertex in  $C$ has none or exactly two neighbours in $S$ (this is linear time checkable). 
Else (if $G$ is  split), sets $A_1$ to $A_{\ell}$ defined in Section 3 must be found. When $\ell=1$ and $|A_1|>1$, we have to check that every vertex in the clique $C$ has two neighbors that are not adjacent to other vertices in $C$, to conclude that $G$ is $\{2\}$-Roman (corresponding to a well labelled spider obtained by applying R3). When $\ell \geq 3$ and $|A_1|=1$, we have to check if $|A_2|=2$ and $B_2\cap C_1=\emptyset $ to conclude that $G$ is $\{2\}$-Roman.
\end{proof}

\section{Conclusion and future work}\label{sec:conclusions}

The class of $\{2\}$-Roman graphs is relatively new. The $\{2\}$-Roman graph recognition problem was posed in a very recent work, polinomially solved only for middle graphs and with open complexity for general graphs \cite{vStorgel2025}.
To the best of our knowledge and as reported in the cited literature, no
other study is known for the class of $\{2\}$-Roman graphs. 

In this work we have contributed to the study of this open problem.

Future research may explore graph classes for which their modular decomposition is known, beyond those graph classes with a limited number of 4-paths, such as  $P_4$-laden graphs and extended $P_4$-laden graphs. 

\section*{Acknowledgements}
This work was supported by  Consejo Nacional de Invesigaciones Científicas y Técnicas through project PIP 0227, and Universidad Nacional de Rosario through project 80020220700042UR. 

\bibliographystyle{abbrv} 
\bibliography{biblio}

\end{document}